\documentclass[12pt]{amsart}
\usepackage{amssymb}
\usepackage{t1enc}
\usepackage[mathscr]{eucal}
\usepackage{parskip}
\usepackage{xcolor}
\usepackage{tikz}
\usepackage{tikz-cd}
\usepackage[hidelinks]{hyperref}
\numberwithin{equation}{section}
\usepackage[margin=2.9cm]{geometry}
\usepackage{booktabs}
\usepackage{enumerate}
\usepackage{bm}
\usepackage{xfrac}

\usepackage{listings}
\theoremstyle{plain}
\newtheorem{thm}{Theorem}[section]
\newtheorem{lem}[thm]{Lemma}
\newtheorem{cor}[thm]{Corollary}
\newtheorem{prop}[thm]{Proposition}
\newtheorem{conj}[thm]{Conjecture}

 \theoremstyle{definition}
\newtheorem{defn}[thm]{Definition}
\newtheorem{rem}[thm]{Remark}
\newtheorem{ex}[thm]{Example}

\newcommand{\disc}{\mathrm{disc}}

\definecolor{mauve}{HTML}{8B0000}
\definecolor{dkgreen}{HTML}{008800}
\usepackage{seqsplit} 
\makeatletter
\@namedef{subjclassname@2020}{\textup{2020} Mathematics Subject Classification}
\makeatother

\begin{document}
%%%%%%%%%%%%%%%%%%%%%%%%%%%%%%%%%%%%%%%%
\title{Monogenic fields of cryptographic size}

\author[Adhikari]{Swechchha Adhikari}
\address{Department of Mathematics \\ Brigham Young University}
\email{adhikar6@byu.edu}

\author[Plott]{Daphne Plott}
\email{daphneplott@gmail.com}

\author[Torgersen]{Parker Torgersen}
\email{parker.23.torgersen@gmail.com}

\subjclass[2020]{Primary: 11R09; Secondary: 11R04, 11R21}
%%%%%%%%%%%%%%%%%%%%%%%%%%%%%%%%%%%%%%%%

\begin{abstract}
For a monic irreducible $f\in\mathbb{Z}[x]$ of degree $n$ and an integer $c$, we study the palindromic transform $F(x)=x^n f(x+x^{-1}+c),$ which produces a polynomial of degree $2n$. We give a discriminant formula $\disc(F)=f(c+2)f(c-2)\disc(f)^2,$ sufficient conditions for the irreducibility of $F$, and a criterion showing that $F$ is monogenic when $f$ is monogenic and $f(c+2)f(c-2)$ is squarefree, together with a matching non-monogenicity criterion. Iterating the transform yields monogenic number fields of degree $2^k n$ from a fixed base. As an explicit example, we construct a monogenic field of degree $512$ from $x^4+x+1$.
\end{abstract}

\maketitle

\section{Introduction}
A number field whose ring of integers is generated by a single element is called \emph{monogenic}. Classifying all monogenic number fields has been an open problem in number theory since Hasse \cite{Hasse1963Zahlentheorie} proposed it in the 1960s. While the complete characterization still remains unknown, Gaál provides some classification for number fields of lower degree in \cite{Gaal}, and several works also construct explicit infinite families of such number fields in small degrees \cite{Lavallee2005Dihedral,Spearman2006PSL25,Lavallee2012PSL27}. In arbitrary degree, these fields are known to be abundant: Kedlaya \cite{Kedlaya2012Squarefree} shows that for every $n \ge 2$ there exist infinitely many monogenic degree-$n$ fields with Galois group $S_n$; Jones \cite{jones2019brief} gives infinite families $x^n + A(Bx+1)^m$ that are monogenic for infinitely many $A$; and for pure fields Nguyen-Dang and Hung \cite{nguyen_dang_hung_2026} give a criterion implying, for instance, that $x^n - \operatorname{rad}(n)$ is monogenic for every $n \ge 2$. In each of these the degree is fixed within the family. Our interest is instead in an iterable mechanism that reaches arbitrarily large degrees, such as Castillo's construction~\cite{castillo2022dynamical}, obtained by iterating a fixed quadratic $t^2-\nu$.

Constructing monogenic fields at cryptographic sizes is motivated in part by lattice-based cryptography, specifically through the ring learning with errors (RLWE) problem \cite{lyubashevsky_peikert_regev_2010}. We refer to number fields of degree \(2^9\) or higher, as used in RLWE-based schemes such as NewHope \cite{newhope}, as cryptographic size. Eisentr\"ager, Hallgren, and Lauter \cite{eisentrager2014weak} list six conditions that are jointly sufficient to produce an attack on the search version of RLWE; monogenicity is one of them, required for the reduction from RLWE to PLWE. Explicit monogenic constructions at cryptographic sizes are therefore a step toward understanding which rings may give rise to weak RLWE instances.

Grossman and Kim \cite{Grossman2015AnIF} found an infinite family of non-abelian monogenic number fields of degree 6. We generalize their construction to large degrees of the form $2^kn$ through an explicit palindromic transform.

\begin{defn}
    A polynomial $f(x)$ is \emph{palindromic} if $x^nf(\tfrac1 x )=f(x)$, where $n$ is the degree of $f(x)$. 
\end{defn}

An intuitive definition is that when the coefficients of $f(x)$ are listed in order, including any coefficients equal to zero, this list stays the same when reversing the order. The polynomial $f(x) = x^3 + 2x^2 + 2x + 1$ is palindromic, but $f(x) = x^3 + 3x^2 + 2x  + 1$ is not.

Given a monic irreducible $f \in \mathbb{Z}[x]$ of degree $n$ and an integer $c$, define \[F(x) = T_c(f)(x) = x^n f\!\left(\frac{1}{x} + x + c\right).\] 
This produces a palindromic polynomial $F(x)$ of degree $2n$. To construct monogenic number fields at large sizes, we
\begin{itemize}
    \item start with a small base polynomial $f$ of degree $n$ where irreducibility and monogenicity can be verified directly,
    \item show that the transformation preserves the properties of irreducibility and monogenicity under certain criteria, and
    \item iterate the transform, verifying these conditions at each step, to obtain monogenic fields of degrees $n, 2n, 4n, \ldots,2^kn$ for each $k$.
\end{itemize}

Our first result gives sufficient conditions for $F(x)$ to be irreducible over
$\mathbb{Q}$.

\begin{thm}\label{theorem:main-irreducibility}
    Let $f(x) \in \mathbb{Z}[x]$ be an irreducible polynomial of degree $n$. Let $c$ be any integer. Construct $F(x) = x^nf(\tfrac 1 x + x + c)$.

    \begin{enumerate}
    \item If either $|F(1)|$ or $|F(-1)|$ is not a perfect square, then $F(x)$ is irreducible.
    \item If $F(1)$ and the middle coefficient of $F(x)$ (the coefficient of $x^n$) have opposite signs, then $F(x)$ is irreducible.
    \item If the middle coefficient of $F(x)$ is $0$ or $\pm 1$, then $F(x)$ is irreducible.
    \end{enumerate}
\end{thm}

The next theorem gives sufficient conditions for $F(x)$ to be monogenic. The main mechanism behind it is the discriminant identity
\[
    \disc(F) = f(c+2)\,f(c-2)\,\disc(f)^2,
\]
proved in Theorem~\ref{thm:disc}. Recall that an integer is squarefree if it is not divisible by the square of any prime. We show that the monogenicity of $F(x)$ reduces to a squarefreeness condition on $f(c+2)f(c-2)$.

\begin{thm}
\label{thm:monogenic}
Let $f(x)\in\mathbb{Z}[x]$ be monic, irreducible, and monogenic. Let $c\in\mathbb{Z}$ and define $F(x)=x^n\,f\!\left(x+\frac{1}{x}+c\right)$. Suppose that $\beta$ is a root of $F(x)$. If $f(c+2)f(c-2)$ is squarefree and $F(x)$ is irreducible, then $\mathcal{O}_{\mathbb{Q}(\beta)}=\mathbb{Z}[\beta]$; in particular $K=\mathbb{Q}(\beta)$ is monogenic.
\end{thm}

The following result shows that non-monogenicity is also preserved under this transformation.

\begin{thm}
\label{thm:nonmonogenic}
    Let $f(x)\in\mathbb{Z}[x]$ be monic and irreducible. Let $c\in\mathbb{Z}$ and define
$F(x)=x^n\,f\!\left(x+\frac{1}{x}+c\right).$ Let $\alpha$ be a root of $f(x)$ and $\beta$ be a root of $F(x)$. If $F$ is irreducible and $\mathbb{Z}[\alpha]\neq\mathcal{O}_{\mathbb{Q}(\alpha)}$, then $\mathbb{Z}[\beta]\neq\mathcal{O}_{\mathbb{Q}(\beta)}$.
\end{thm}

\begin{rem}
Assuming the ABC conjecture, iterating the transform starting from a suitable base polynomial produces infinitely many monogenic number fields of degree $2^k n$ for each $k$, as discussed in Section~\ref{sec:open}.
\end{rem}

When $c=0$, our transform is the reciprocal substitution $x \mapsto x + x^{-1}$ and in that case many of the tools we use are well-studied in the literature. Jones~\cite{jones2021infinite} constructed infinite families of palindromic monogenic polynomials; following this line of work, Barman, Narode, and Wagh~\cite{BNW26} treated the monogenicity of palindromic polynomials in general. At $c=0$ our discriminant identity, monogenicity criterion, and non-monogenicity result reduce to Lemma~2.1, Theorem~1.5, and Proposition~1.8 of~\cite{BNW26}, respectively. Our work relates to these ideas in two ways. First, we use an arbitrary $c$ in the transform $x + x^{-1} +c$, which lets us work from a simpler base polynomial and gives us a free parameter to make $f(c+2)f(c-2)$ squarefree. Second, we iterate the transform itself on a monic irreducible base. Iterating the transform can then give monogenic fields of degrees $n, 2n, 4n, \dots, 2^k n$. Concretely, we reach a monogenic field of degree $512$ from the base $x^4 + x + 1$ (see Section~\ref{sec:example}).

In the rest of this paper, we first discuss general properties of palindromic polynomials and of the transformation $T_c$ (Section~\ref{sec:transformation}). We then determine when $F(x)$ is irreducible, by constructing a reverse transformation (Section~\ref{sec:irreducibility}). Next we examine when $F(x)$ is monogenic (Section~\ref{sec:monogenicity}) by finding a formula for the discriminant of $F(x)$ and using the tower formula to obtain factors of the discriminant of the number field generated by $F(x)$. Finally, we work through an example of applying our criteria to a polynomial that has undergone the transformation several times (Section~\ref{sec:example}).

\subsection*{Acknowledgments}
This research was supported by the College of Computational, Mathematical, and Physical Sciences at BYU. The authors thank Nick Andersen and Kyle Pratt for their helpful comments.

\section{Basic Properties of the Palindromic Transform}
\label{sec:transformation}

In this section, we first discuss properties of palindromic polynomials, then the properties of the transformation $T_c$, and define the notation we use in the paper.

\begin{lem}\label{product of palindromic polynomials is palindromic}
    The product of two palindromic polynomials is palindromic.
\end{lem}
\begin{proof}
    Assume that $f(x)$ and $g(x)$ are palindromic, with $n$ the degree of $f(x)$, and $m$ the degree of $g(x)$. So, $f(x) = x^nf(\tfrac 1 x)$ and $g(x) = x^mg(\tfrac 1 x)$. Let $h(x) = f(x)g(x)$. Note that the degree of $h$ is $n+m$. We find $h(x) = f(x) g(x) = x^nf(\tfrac1 x) x^mg(\tfrac 1 x) = x^{n+m}f(\tfrac 1 x)g(\tfrac 1 x) = x^{n+m}h(\tfrac 1 x)$. Hence $h$ is palindromic of degree $n+m$.
\end{proof}

\begin{lem}\label{Odd-degree palindromic polynomials contain factor (x+1)}
    Any odd degree palindromic polynomial contains the factor $(x+1)$.
\end{lem}
\begin{proof}
    Let $p(x)$ be a palindromic polynomial of odd degree $2k+1$. Because it has odd degree, the number of terms in total is even. We write it as \[
    p(x)= \sum_{i=0}^k a_i(x^i + x^{2k+1 - i}). \]We note that $i$ and $2k+1 - i$ will always have opposite parity. We see \[p(-1) = \sum_{i=0}^ka_i(-1 + 1) = 0\]so $(x+1)$ is a divisor of $p(x)$.
\end{proof}

\begin{defn}
Let $c$ be an integer. Define the transformation $T$ by $T(f) = x^n f(x + \frac{1}{x} + c),$ where $f(x)$ is a polynomial. When the value of $c$ is not clear from context, we write $T_c$ to indicate the dependence on $c$. We use the notation $T_c^2(f) = T_c(T_c(f))$, and similarly for higher powers.
\end{defn}

\begin{lem}
Fix $c \in \mathbb Z$. Let $f\in\mathbb{Z}[x]$ have degree $n$. Define $F_c(x)=T_c(f)$. Then $F_c$ is a palindromic polynomial of degree $2n$.
\end{lem}
\begin{proof}
    Since $f$ has degree $n$, substituting $x+\frac{1}{x}+c$ produces terms of highest degree $x^n$ and $x^{-n}$, and multiplying by $x^n$ yields a polynomial of degree $2n$. Also, $x^{2n}F_c(\tfrac 1 x)=x^n f\!\left(x+\frac{1}{x}+c\right)=F_c(x),$ so $F_c$ is palindromic.
\end{proof}

\begin{lem}\label{Tc Injective}
    Fix $c \in \mathbb{Z}$. Then $T_c$ is injective from $\mathbb{Z}[x]$ to $\mathbb{Z}[x]$.
\end{lem}

\begin{proof}
Let $f_1, f_2 \in \mathbb{Z}[x]$.
    Assume $T_c(f_1) = T_c(f_2)$. Because $T_c$ always results in a polynomial with double the degree of the original polynomial, we see the degrees of $f_1$ and $f_2$ are equal. Let $n$ be this degree. So, 
    \begin{align*}
        x^nf_1(x + \tfrac 1 x +c) =& \ x^nf_2(x + \tfrac 1 x +c) \\
        f_1(x+\tfrac 1 x +c) =& \ f_2(x + \tfrac 1 x +c).
    \end{align*}
        Let $a \in (-\infty, c-2] \cup [c+2,\infty)$. Then we can find an $x$ such that $a = x + \tfrac 1 x + c$. So, $f_1(a) = f_2(a)$ along that interval. Now, set $R = f_1-f_2$. We know that for all $a \in (-\infty, c-2] \cup [c+2,\infty)$, $R(a) = 0$. Since $R$ has infinitely many zeros, it must be the zero polynomial. So, $f_1 = f_2$.
\end{proof}

Our next result shows that $T_c$ is a multiplicative homomorphism.
\begin{lem}\label{Tf1f2) = T(f1)T(f2)}
    Fix $c \in \mathbb{Z}$. Then, $T_c(f_1)T_c(f_2) = T_c(f_1 f_2)$.
\end{lem}
\begin{proof}
    Let $n_i$ be the degree of $f_i$. Set $f = f_1f_2$ and $n = n_1+n_2$. Note that $f$ has degree $n$. We have
  \[
  x^{n_1}f_1(\tfrac 1 x + x +c)\ x^{n_2}f_2(\tfrac 1 x +x+c) 
  = x^nf(\tfrac 1 x + x +c). \qedhere
  \] 
\end{proof}

\begin{rem}
    For the rest of the paper, we define $F = F_c(x) = T_c(f)$ for some fixed $c \in \mathbb{Z}$. Although $F(x)$ depends on the choice of $c$, we omit it from our notation for simplicity.
\end{rem}

\begin{lem}\label{non-zero roots}
    Fix $c \in \mathbb{Z}$. The polynomial $F(x)$ does not have zero as a root.
\end{lem}
\begin{proof}
Write $f(x) = a_nx^n + a_{n-1}x^{n-1} + \cdots + a_1x + a_0$, and $F(x) = b_{2n}x^{2n} + b_{2n-1}x^{2n-1} + \cdots + b_1x + b_0$. 
    The constant term of $F$ is equal to the leading coefficient of $f$, so $b_0 = a_n$. The leading coefficient of $f$ cannot be zero, so $F(0) = b_0 = a_n \neq 0$, so 0 is not a root.
\end{proof}

\section{Irreducibility}
\label{sec:irreducibility}
To ensure that $\mathbb{Q}[x]/(F(x))$ is a number field, the transformed polynomial $F(x)$ must be irreducible over $\mathbb{Q}$. One might expect that the irreducibility of $f(x)$ over $\mathbb{Q}$ would also imply the irreducibility of $F(x)$. However, that is not the case in general, as shown by the following example.

\begin{ex}
    Let $c = 2$ and $f(x) = x^3 -x +4$. Define $F(x) = T_c(f)$. Then,
    \[F(x) = (x^3 + 2x^2 +4x + 1)(x^3 + 4x^2 + 2x  + 1).\]
    One can easily check that $f(x)$ is irreducible, but $F(x)$ is not. 
\end{ex}

The strategy for determining when $F(x)$ is irreducible is as follows. Since $F(x)$ is palindromic, we first show that any factorization of $F(x)$ must either be a product of palindromic factors or a product $g(x)\cdot x^mg(\tfrac 1 x)$ for some non-palindromic polynomial $g$. We then rule out each type in turn. The palindromic case is handled by constructing a reverse transformation, which shows that palindromic factors of $F$ would give factors of $f$, contradicting irreducibility. The $g(x) \cdot x^mg(\tfrac 1 x)$ case is ruled out by simple criteria involving only $F(1)$ and $F(-1)$.

We first prove that $F$ can only be irreducible if $f$ is irreducible. 

\begin{lem}
If $f(x)$ is reducible, then $F(x)$ is reducible. 
\end{lem}

\begin{proof}
    By Lemma~\ref{Tf1f2) = T(f1)T(f2)}, if $f = f_1f_2$, then $F = T_c(f_1)T_c(f_2)$.
\end{proof}

When $f(x)$ is reducible, then one factorization of $F(x)$ reflects applying the transformation individually to each factor of $f(x)$. However, these may not be the only factors, as $T_c(f_i)$ may be reducible. If $f(x)$ is irreducible, the factors of $F(x)$ follow a different pattern, as seen in the above example and discussed below.

\begin{defn}
    For a polynomial $g$ with degree $m$, define $g^{rev}$ as $g^{rev}(x) = x^mg(\tfrac 1 x)$.
\end{defn}

\begin{ex}
    Given the polynomial $g(x) = 2x^3 + 7x^2 + 9$, $g^{rev}(x) = 9x^3 + 7x + 2$.
\end{ex}

\begin{lem}
\label{grev}
    Given a palindromic polynomial $p(x)$, with degree $k$, if $p$ has some factor $g$, it also has the factor $g^{rev}$. Additionally, $g \cdot g^{rev}$ is palindromic.
\end{lem}

\begin{proof}
    If $g(x)$ is a palindromic factor of $p(x)$, then $g = g^{rev}$, and there is nothing to prove.

    Assume that $g(x)$ is a non-palindromic factor of $p(x)$.  Let $\gamma_i$ be the roots of $g(x)$, and let $m$ be the degree of $g(x)$. Then, $\gamma_i$ is also a root of $p(x)$. Because $p(x)$ is palindromic, $\frac{1}{\gamma_i}$ are also roots of $p(x)$, as $0 = p(\gamma_i) = \gamma_i^{k}p(\tfrac{1}{\gamma_i})$. By Lemma \ref{non-zero roots}, $\gamma_i \neq 0$, so $p(\tfrac{1}{\gamma_i}) = 0$. Then, the polynomial $g^{rev}(x) = x^mg(\tfrac{1}{x})$ has roots $\frac{1}{\gamma_i}$. Since $p(x)$ has all the roots of $g^{rev}$, $g^{rev}$ must be a factor of $p(x)$.

    We have \[g(x)\ g^{rev}(x) = g(x) \ x^mg(\tfrac1 x)  = x^{2m} \tfrac{1}{x^m}g(x)\ g(\tfrac1 x) = x^{2m} g(\tfrac 1 x)\ g^{rev}(\tfrac1 x), \] hence $g\cdot g^{rev}$ is palindromic. 
\end{proof}

Note that not every non-palindromic factor $g$ will satisfy $p = g\cdot g^{rev}$, but for any non-palindromic, irreducible factor $g$, we have $\pm gg^{rev}$ is a factor of $p$.

\begin{lem}
    For any non-palindromic, irreducible factor $g$ of a palindromic polynomial $p$, we have $\pm gg^{rev}$ is a factor of $p$.
\end{lem}

\begin{proof}
Since $g \mid p$, we also have $g^{rev} \mid p$. If $g$ and $g^{rev}$ are coprime, then $gg^{rev}$ is a factor of $p$. Thus, it remains to consider the case where $g$ and $g^{rev}$ are not coprime. Since $g$ is irreducible, the only possibilities are $g = x+1$ or $g= x -1$. If $g = x+1$, then $g$ is palindromic. If $g = x-1$, then $g^{rev} = 1 -x$. We must have $p(x) = gh$ for some polynomial $h$, since $g$ is not palindromic. Then, \[p(x) = x^k p(\tfrac 1  x) = x^{k_1} g(\tfrac 1 x) \cdot x ^{k_2}h(\tfrac 1 x),\] where $k$ is the degree of $p$, and $k_1, k_2$ are the degrees of $g, h$ respectively. So $(x-1)h = -(x-1)h^{rev}$, or $h = -h^{rev}$. Then $h(1) = -(1)^{k_2}h(\tfrac 1 1) = -h(1)$, so $h(1) = 0$. We can then factor out $(x-1)$ from $h$ such that there exists a polynomial $q(x)$ that satisfies $p(x) = (x-1)(x-1)q(x)$. So $\pm gg^{rev} | p(x)$.
\end{proof}

We now classify the possible factorizations of a palindromic polynomial.

\begin{lem}
    \label{factorization}
    A reducible palindromic polynomial of even degree factors in one of the following ways:
    \begin{enumerate}
    \item $F = p_1\cdots p_n$ where each $p_i$ is a palindromic polynomial of even degree or
    \item $F = \pm g\cdot g^{rev}$, where $g$ is a non-palindromic polynomial.
    \end{enumerate}
\end{lem}
\begin{proof}
    Assume that $F$ is reducible. We work by cases based on whether the irreducible factors are palindromic or not.
    
    \textit{Case 1: The factors of $F(x)$ are all palindromic.} Any odd degree palindromic polynomial contains the factor $(x+1)$ (\ref{Odd-degree palindromic polynomials contain factor (x+1)}). So the only irreducible odd degree palindromic polynomial is $(x+1)$. There must be an even number of these because $F$ has even degree. As such, we can combine any of those factors together to create an even degree polynomial. This is option (1).
    
    \textit{Case 2: The factors of $F(x)$ are not all palindromic.} By Lemma~\ref{grev}, if $F(x)$ has a non-palindromic factor $g$, it also has the factor $g^{rev}$, and the product $g \cdot g^{rev}$ is palindromic. If $F$ factors into a single such pair $g \cdot g^{rev}$, this is option (2). If $F$ has multiple such pairs, or a single pair and some palindromic factors, we can combine each pair into a palindromic polynomial $p_i$, giving option (1).
    
    In all cases, $F$ falls into option (1) or option (2).
\end{proof}

The following criteria determine when $F$ cannot factor as $g\cdot g^{rev}$. We give the proof from \cite{CafureCesaratto17} for the sake of completeness.

\begin{lem}\label{gxg^rev factoring criteria}
    Let $F(x)$ be a palindromic polynomial with degree $2n$.
    \begin{enumerate}
    \item If $|F(1)|$ or $|F(-1)|$ is not a perfect square, then $F(x)$ cannot factor as $\pm g(x)g^{rev}(x)$.
    \item If $F(1)$ and the coefficient of $x^n$ of $F(x)$---the middle coefficient of $F(x)$---have opposite signs, then $F(x)$ cannot factor as $\pm g(x)g^{rev}(x)$.
    \item If the middle coefficient of $F(x)$ is $0$ or $\pm 1$, then $F(x)$ cannot factor as $\pm g(x)g^{rev}(x)$.
    \end{enumerate}
\end{lem}

\begin{proof}
We work contrapositively. Assume that $F(x) = \pm g(x)\cdot g^{rev}(x)$. For (1), we find that \[|F(\pm1)| = | g(\pm1)g^{rev}(\pm1) |=| g(\pm1)(\pm1)^mg(\pm1)|=|g(\pm1)^2|.\]

For (2), write $g(x) = a_0 + a_1x + \cdots + a_mx^m$. The middle 
coefficient of $F(x) = \pm g \cdot g^{rev}$ is $\pm(a_0^2 + a_1^2 + \cdots + a_m^2)$, where the sign is the same $\pm$ as in $F(x) = \pm g \cdot g^{rev}$. Since $F(1) = \pm g(1)^2$ carries the same sign, then the coefficient of $x^n$ of $F(x)$,  which is the middle coefficient of $F(x)$, and $F(1)$ have the same sign.

For (3), writing $g(x)$ the same way, the middle coefficient of $F(x)$ is $\pm(a_0^2 + a_1^2 + \cdots + a_m^2)$. If this is $0$, then $g(x)$ is the zero polynomial. If this is $\pm 1$, then $g(x)$, which must have integer coefficients, can only have one non-zero coefficient. We can write it as $g(x) = \pm x^k$ for some $k$, and so $g^{rev}(x)=\pm1$. $F(x)$ cannot equal $g \cdot g^{rev}$ because $F(x)=\pm x^k$ is not palindromic.
\end{proof}

If $F$ does not factor as $g \cdot g^{rev}$, then it must either be irreducible or factor as a product of palindromic polynomials. To rule out the latter case, we construct a reverse transformation 
for palindromic polynomials.

\begin{lem}
    Fix $c \in \mathbb{Z}$. Consider the recursively defined sequence of polynomials $h_k$, where
    \begin{gather*}
        h_0(x) = 2 \\
h_1(x) = x- c\\
h_k(x) = (x-c)\cdot h_{k-1}(x) - h_{k-2}(x)
    \text{ for all } k \ge 2. 
    \end{gather*}
    
    Then, each $h_k$ is a polynomial with integer coefficients, and $h_k(\tfrac1 x + x + c) = \frac{1}{x^k} + x^k$.
\end{lem}

\begin{proof}
    We proceed by induction on $k$. For $k = 0$, we see $h_0(\tfrac1 x + x + c) = \tfrac{1}{x^0} + x^0$. For $k = 1$, we see $h_1(\tfrac{1}{x} + x + c) = \tfrac 1 x + x$. Assume for some $n \in \mathbb{N}$ that $h_n(\tfrac{1}{x} + x + c) = \tfrac{1}{x^n} + x^n$ and that $h_{n-1}(\tfrac 1 x + x + c) = \tfrac{1}{x^{n-1}}+x^{n-1}$. Assume $h_n$ and $h_{n-1}$ are polynomials with integer coefficients. We find
    \[
        h_{n+1}(\tfrac1 x + x + c) =(\tfrac 1 x + x + c - c)\ h_n(\tfrac1 x +x+c) - h_{n-1}(\tfrac 1 x + x+c) \\
        = \tfrac{1}{x^{n+1}}+x^{n+1}.
\]

    Because $h_{n+1}$ is the product and sum of polynomials with integer coefficients, then $h_{n+1}$ is also a polynomial with integer coefficients.
\end{proof}

\begin{rem}
    The polynomials $h_k$ are shifted Dickson polynomials of the first kind. Specifically, $h_k(x) = D_k(x-c, 1)$ where $D_k(x, \alpha)$ denotes the Dickson polynomials with parameter $\alpha$ (see \cite{LidlMullenTurnwald1993}).
\end{rem}

The following lemma constructs the reverse transformation.

\begin{lem}\label{Reverse transformation}
    Fix $c \in \mathbb{Z}$. Given a palindromic polynomial $F(x)$ of degree $2n$, we write
    \[F(x) = a_0x^n + \sum _{k=1} ^{n} a_k(x^{n-k} +x^{n+k}) \]
    and construct 
    \[
    f(x) = a_0 + \sum_{k=1}^{n}a_kh_k(x)
    \]
    where each $h_k$ is determined by the given $c$ value. Then, $x^nf(\tfrac 1 x + x + c) = F(x).$
\end{lem}
\begin{proof}
    We directly apply the transformation $T_c$ to the constructed $f(x)$. We find
\[
        x^nf(\tfrac1 x +x + c) =  x^n \left( a_0 + \sum_{k=1}^n a_kh_k(\tfrac1 x + x + c) \right) 
        = F(x).
\]

Because $F(x)$ is palindromic of even degree, it has a center term $a_0x^n$, and the coefficients of $x^{n-k}$ and $x^{n+k}$ are equal, so the representation above is valid.
\end{proof}

\begin{prop}\label{Irreducible => no palindromic factors}
    If $f(x)$ is irreducible, then $F(x)$ cannot have any palindromic factors. 
\end{prop}

\begin{proof}
    We work contrapositively. Suppose $F(x)$ has palindromic factors. For simplicity, write this as only two factors, $F_1$ and $F_2$, both even-degree palindromic polynomials, but not necessarily irreducible. To see they must be even degree, we know that $(x+1)$ is a factor of every odd degree palindromic polynomial (\ref{Odd-degree palindromic polynomials contain factor (x+1)}), so we can split any odd degree factors into $(x+1)\cdot F_i$ where $F_i$ has to be of even degree. Because $F(x)$ has even degree, it must have an even number of these $x+1$ factors, which we can multiply together to become an even degree polynomial. By Lemma~\ref{product of palindromic polynomials is palindromic}, the product is palindromic.

    According to the reverse transformation in Lemma \ref{Reverse transformation}, we know we can create $f_1(x)$ and $f_2(x)$ such that $T_c(f_1) = F_1$ and $T_c(f_2) = F_2$. We then see that
\[F = F_1 F_2 = T_c(f_1)\ T_c(f_2) = T_c(f_1f_2).\]
    By Lemma \ref{Tc Injective}, because our transformation is injective, we know that $f_1f_2 = f$. So $f(x)$ is reducible.
\end{proof}

We now combine this result with Lemma~\ref{gxg^rev factoring criteria} to prove Theorem \ref{theorem:main-irreducibility}.

\begin{proof}[Proof of Theorem \ref{theorem:main-irreducibility}]
    By Proposition \ref{Irreducible => no palindromic factors}, since $f(x)$ is irreducible, then $F(x)$ will not factor as palindromic polynomials. According to Lemma \ref{gxg^rev factoring criteria}, if criteria (1), (2) or (3) is fulfilled, then $F(x)$ doesn't factor as $g\cdot g^{rev}$. Hence $F(x)$ must be irreducible.
\end{proof}

\begin{rem}
\label{rem:eval}
Observe that $f(c+2) = F(1)$ and $f(c-2) = \pm F(-1)$. Although $f(x)$ may have an odd degree power, any further $T^k(f)$ will have even degree. Since $c \in \mathbb{Z}$ and $f$ has integer coefficients, $F$ has 
integer coefficients. Checking whether $|F(1)|$ or $|F(-1)|$ is a 
perfect square requires only evaluating the polynomials at two integers, and checking to see if those are perfect squares. These operations remain
feasible even at cryptographic degree.
\end{rem}

\section{Monogenicity}
\label{sec:monogenicity}
In this section, we show that if $f$ generates a monogenic number field, then so does $F$ under a simple squarefreeness condition. Throughout the monogenicity discussion, we assume that $F = T_c(f)$ is irreducible; recall that irreducibility can be certified by Theorem ~\ref{theorem:main-irreducibility}. Our main objects of study in this section are the number field $K$, its ring of integers $\mathcal{O}_K$, and its discriminant $\disc(K)$. We assume some familiarity with these objects and their basic properties; further background and details can be found in \cite{neukirch1999algebraische}.

\begin{defn}
    Let $f(x)\in \mathbb{Z}[x]$ be a monic polynomial of degree $n$ with roots $\alpha_1,\dots,\alpha_n$. The discriminant of $f$ is
    \[\disc(f) = \prod_{1\le i<j\le n} (\alpha_i-\alpha_j)^2.\]
\end{defn}

We use the following discriminant criterion.
\begin{thm}
\label{disctest}
Let $f(x)$ be a monic irreducible polynomial with integer coefficients, $\alpha$ a root of $f$, and $K = \mathbb{Q}(\alpha)$. Then $\mathcal{O}_K = \mathbb{Z}[\alpha]$ if and only if $\disc(f) = \disc(K)$. 
In this case $\alpha$ generates $\mathcal{O}_K$, so $K$ is monogenic.
\end{thm}
\begin{proof}
The discriminant of $f(x)$ equals the discriminant of $\mathbb{Z}[\alpha]$, 
where $\alpha$ is a root of $f(x)$. Since $\mathcal{O}_K$ contains 
$\mathbb{Z}[\alpha]$, the index formula (see \cite{neukirch1999algebraische} Proposition 2.12) gives
$$\disc(\mathbb{Z}[\alpha]) = \disc(K)\,[\mathcal{O}_K : \mathbb{Z}[\alpha]]^2.$$
If $\disc(f) = \disc(K)$, then the two discriminants are equal, so 
$[\mathcal{O}_K : \mathbb{Z}[\alpha]]^2 = 1$. Hence 
$[\mathcal{O}_K : \mathbb{Z}[\alpha]] = 1$, so $\mathcal{O}_K = \mathbb{Z}[\alpha]$ 
and $\alpha$ generates $\mathcal{O}_K$.

Conversely, if $\mathcal{O}_K = \mathbb{Z}[\alpha]$, then 
$[\mathcal{O}_K : \mathbb{Z}[\alpha]] = 1$, and the same index formula gives $\disc(f) = \disc(K)$.
\end{proof}
\begin{rem}
Throughout the paper, when we say $f$ is \emph{monogenic} we mean 
$\disc(f) = \disc(K)$, i.e.\ that the root $\alpha$ of $f$ itself generates 
$\mathcal{O}_K$.
\end{rem}

Observe that when $f(x)$ is monic, $F(x) = x^nf\!\left(x + \frac{1}{x} + c\right)$ is also monic, with leading term $x^{2n}$. We consider only monic polynomials.

To compute the discriminant of $F(x)$, we first determine its roots.

\begin{lem}
\label{roots}
Let $f(x)\in \mathbb{Z}[x]$ be a monic polynomial with roots $\alpha_1,\dots,\alpha_n$ and define $F(x)=x^n f\!\left(x+\frac{1}{x}+c\right).$ Then the roots of $F(x)$ are
$$\frac{(\alpha_i-c)\pm\sqrt{(\alpha_i-c)^2-4}}{2},
\qquad i=1,\dots,n.$$
\end{lem}

\begin{proof}
First write $f(x)=\prod_{i=1}^n (x-\alpha_i).$ Then we have
\[F(x)=x^n f\!\left(x+\frac{1}{x}+c\right)
= \prod_{i=1}^n x\!\left(x+\frac{1}{x}+c-\alpha_i\right) = \prod_{i=1}^n x^2+(c-\alpha_i)x+1.\] Hence the roots of $F(x)$ are the roots of the quadratics $x^2+(c-\alpha_i)x+1=0 \text{ for each } i \in \{1,\dots,n\}.$ The result then follows by applying the quadratic formula.
\end{proof}

Using the explicit description of the roots of $F(x)$, we now compute its discriminant.

\begin{thm}
\label{thm:disc}
Let $f(x)\in\mathbb{Z}[x]$ be monic and separable of degree $n$, and let $c\in\mathbb{Z}$.
Then
$$\disc(F)=f(c+2)f(c-2)\disc(f)^2.$$
\end{thm}
\begin{proof}
By Lemma~\ref{roots}, the roots of $F$ are
$$\beta_{i,\pm}=\frac{(\alpha_i-c)\pm\sqrt{(\alpha_i-c)^2-4}}{2},\qquad i=1,\dots,n,$$
where $\alpha_1,\dots,\alpha_n$ are the roots of $f$. Since $F$ is monic of degree
$2n$, its discriminant is the product of $(\beta_{i,\epsilon}-\beta_{j,\delta})^2$
over all unordered pairs of distinct roots. These pairs are of two types.

\textit{Same index.} The pair $(\beta_{i,+},\beta_{i,-})$ arising from a single root
$\alpha_i$ of $f$ contributes
$$(\beta_{i,+}-\beta_{i,-})^2=(\alpha_i-c)^2-4=\bigl(\alpha_i-(c+2)\bigr)\bigl(\alpha_i-(c-2)\bigr).$$
Taking the product over $i$ and using $\prod_{i=1}^n(\alpha_i-t)=(-1)^n f(t)$,
$$\prod_{i=1}^n(\beta_{i,+}-\beta_{i,-})^2=(-1)^n f(c+2)\cdot(-1)^n f(c-2)=f(c+2)f(c-2).$$

\textit{Distinct indices.} Write $A_i=\alpha_i-c$, so that the roots of $F$ are
$\beta_{i,\pm}=\tfrac12\bigl(A_i\pm\sqrt{A_i^2-4}\bigr)$, and set $B_i=A_i^2-4$. Fix
$i\neq j$. The four differences $\beta_{i,\epsilon}-\beta_{j,\delta}$ all have the form
$\tfrac12\bigl((A_i-A_j)\pm\sqrt{B_i}\pm\sqrt{B_j}\bigr)$, so pairing them by the sign of
$\sqrt{B_j}$ and applying $(a-b)(a+b)=a^2-b^2$ with $a=A_i-A_j$ gives
$$\prod_{\epsilon,\delta\in\{+,-\}}(\beta_{i,\epsilon}-\beta_{j,\delta})
=\frac{1}{16}\Bigl((A_i-A_j)^2-\bigl(\sqrt{B_i}-\sqrt{B_j}\bigr)^2\Bigr)
\Bigl((A_i-A_j)^2-\bigl(\sqrt{B_i}+\sqrt{B_j}\bigr)^2\Bigr).$$
Applying the difference of squares once more, this equals
$\tfrac{1}{16}\bigl(((A_i-A_j)^2-B_i-B_j)^2-4B_iB_j\bigr)$, and since
$(A_i-A_j)^2-B_i-B_j=-2A_iA_j+8$ we obtain
$$\prod_{\epsilon,\delta\in\{+,-\}}(\beta_{i,\epsilon}-\beta_{j,\delta})
=\frac{1}{16}\Bigl((-2A_iA_j+8)^2-4(A_i^2-4)(A_j^2-4)\Bigr)
=(A_i-A_j)^2=(\alpha_i-\alpha_j)^2.$$
Each such pair contributes its square to $\disc(F)$, so multiplying over all $i<j$,
$$\prod_{i<j}\ \prod_{\epsilon,\delta}(\beta_{i,\epsilon}-\beta_{j,\delta})^2
=\prod_{i<j}(\alpha_i-\alpha_j)^4=\disc(f)^2.$$

Combining the two contributions, $\disc(F)=f(c+2)f(c-2)\cdot\disc(f)^2.$
\end{proof}

To apply Theorem~\ref{disctest}, we relate the discriminant of $\mathbb{Q}[x]/(F)$ to the quantities computable from $f$. We do this using the following general result from \cite{cohen1993computational} (Proposition 4.4.8).

\begin{thm}\label{Conductor-Discriminant Formula}
    For any two number fields $A\subseteq B$ we have disc$(A)^{[B:A]} | $disc$(B)$.
\end{thm}

The following proposition establishes that $\mathbb{Q}(\alpha) \subseteq \mathbb{Q}(\beta)$, where $\beta$ is a root of $F(x)$ and $\alpha=\beta+\beta^{-1}+c$, which allows us to apply Theorem~\ref{Conductor-Discriminant Formula}.

\begin{prop}
\label{prop:subset}
Let $f\in\mathbb{Z}[x]$ be monic and irreducible. If $\beta$ is a root of $F$ and $\alpha=\beta+\beta^{-1}+c$, then $\alpha$ is a root of $f$. In particular, $\mathbb{Q}(\alpha)\subseteq\mathbb{Q}(\beta) \text{ and } [\mathbb{Q}(\beta) : \mathbb{Q}(\alpha)] \in \{1, 2\}.$
\end{prop}
\begin{proof}
    Let $\beta$ be a root of $F(x)$ i.e. $F(\beta) = 0.$ Plugging in $x = \beta$, we get $F(\beta) = \beta ^n f(\beta + \frac{1}{\beta}+c).$ By Lemma \ref{non-zero roots}, $\beta \neq 0$. Hence we must have $f(\beta+ \frac{1}{\beta}+c) = 0.$ Define $\alpha := \beta + \frac{1}{\beta}+c. $ Then $\alpha$ is a root of $f$. Clearly $\beta \in \mathbb{Q}(\beta)$, $\frac{1}{\beta} \in \mathbb{Q}(\beta)$, and $c \in \mathbb{Q}.$ It follows that $\alpha \in \mathbb{Q}(\beta).$ Since $\mathbb{Q}(\alpha)$ is a simple extension, it is by definition the smallest field extension containing both $\mathbb{Q}$ and $\alpha$. Therefore $\mathbb{Q}(\alpha) \subseteq \mathbb{Q}(\beta).$

    From $\alpha = \beta + \beta^{-1} + c$, multiplying through by $\beta$ gives $$\beta^2 - (\alpha - c)\beta + 1 = 0,$$ so $\beta$ is a root of the non-zero polynomial $g(x) = x^2 - (\alpha-c)x + 1 \in \mathbb{Q}(\alpha)[x]$. Since $\deg g = 2$, the minimal polynomial of $\beta$ over $\mathbb{Q}(\alpha)$ divides $g$, giving $[\mathbb{Q}(\beta) : \mathbb{Q}(\alpha)] \in \{1, 2\}$.
    \end{proof}

\begin{cor}
\label{cor:Fdegree}
With the same notation as Proposition~\ref{prop:subset}, suppose \(F(x)\) is irreducible. Then $[\mathbb{Q}(\beta) : \mathbb{Q}(\alpha)] = 2.$
\end{cor}
\begin{proof}
    If $F(x)$ is irreducible, then we have $[\mathbb{Q}(\beta):\mathbb{Q}] = 2n$ and since $[\mathbb{Q}(\alpha):\mathbb{Q}] = n$, by the multiplicativity of field extension degrees (see \cite{dummit2004abstract} Sec. 13.2, Theorem 14) we obtain $[\mathbb{Q}(\beta):\mathbb{Q}(\alpha)] = 2$.
\end{proof}

Theorem~\ref{theorem:main-irreducibility} gives the criteria for when $F$ is irreducible. In that case Corollary~\ref{cor:Fdegree} gives $[\mathbb{Q}(\beta) : \mathbb{Q}(\alpha)] = 2$. Then applying Theorem~\ref{Conductor-Discriminant Formula}, we have
$$\disc(\mathbb{Q}(\alpha))^2 \mid 
\disc(\mathbb{Q}(\beta)).$$
Assuming $f$ is monogenic, we have $\disc(f) = 
\disc(\mathbb{Q}(\alpha))$, so we obtain 
$\disc(f)^2 \mid \disc(\mathbb{Q}(\beta))$, 
and we may write $\disc(\mathbb{Q}(\beta)) = 
\disc(f)^2 \cdot k$ for some integer $k$.

We now compute the discriminant ratio and give a proof of Theorem~\ref{thm:monogenic}. The calculation is
\begin{align*} 
\frac{\disc(F)}{\disc(\mathbb{Q}(\beta))} 
&= \frac{\disc(f)^2 \cdot f(c+2)f(c-2)}{\disc(f)^2 \cdot k} \\ &= \frac{f(c+2)f(c-2)}{k} \\ 
&= [\mathcal{O}_{\mathbb{Q}(\beta)} : \mathbb{Z}[\beta]]^2
\end{align*}

\begin{proof}[Proof of Theorem~\ref{thm:monogenic}]
We have $$[\mathcal{O}_K : \mathbb{Z}[\beta]]^2 = \frac{f(c+2)f(c-2)}{k}.$$ Since the index squared is an integer, it divides $f(c+2)f(c-2)$. Assume $f(c+2)f(c-2)$ is squarefree. Then its only square divisor is $1$.
Hence we must have $[\mathcal{O}_K : \mathbb{Z}[\beta]] = 1,$ so $\mathcal{O}_K = \mathbb{Z}[\beta]$ and the field is monogenic.
\end{proof}

\begin{rem}
Note that $f(c+2)f(c-2)$ will be squarefree when $f(c+2)$ and $f(c-2)$ are individually squarefree, and have a greatest common divisor of $1$.
\end{rem}

We now discuss how to guarantee that $F(x)$ is not monogenic and give a proof of Theorem~\ref{thm:nonmonogenic}. For this, we use the following result from 
\cite{neukirch1999algebraische} (Chapter III, Corollary~2.10), 
which we state here in terms of discriminants using Theorem~2.9 of 
the same reference.~\footnote{The original theorem in the reference is given in terms of the different rather than the discriminant.}

\begin{prop}
\label{tower}
For a tower of number fields $K \subseteq L \subseteq M$, one has
\[\operatorname{disc}(M|K) = \operatorname{disc}(L|K)^{[M:L]} 
\cdot N_{L|K}(\operatorname{disc}(M|L)).\]
\end{prop}

\begin{proof}[Proof of Theorem~\ref{thm:nonmonogenic}]
We work contrapositively. Let $K = \mathbb{Q}(\alpha)$ and $L = \mathbb{Q}(\beta)$. Let $I_\alpha = [\mathcal{O}_K : \mathbb{Z}[\alpha]] \text{ and } I_\beta = [\mathcal{O}_L : \mathbb{Z}[\beta]] = 1.$ Then we have $\disc(f) = I_\alpha^2 \disc(K)$ and $\disc(F) = \disc(L).$ We know that $\disc(F) = \disc(f)^2 \cdot f(c+2)f(c-2),$ so we obtain
$$\disc(L) = \big(I_\alpha^2 \disc(K)\big)^2 \cdot f(c+2)f(c-2).$$
Hence
\begin{equation}
    \disc(L) = I_\alpha^4 \disc(K)^2 \cdot f(c+2)f(c-2).
    \label{eq:disc-L}
\end{equation}
By Proposition ~\ref{tower} applied to the tower $\mathbb{Q} \subseteq K \subseteq L$ with $[L:K]=2$, we also have
\[\disc(L) = N_{K/\mathbb{Q}}(\disc(L/K)) \cdot \disc(K)^2.\]

Since $\mathcal{O}_L = \mathbb{Z}[\beta]$ by assumption, we have 
$\mathbb{Z}[\beta] \subseteq \mathcal{O}_K[\beta] \subseteq \mathcal{O}_L = \mathbb{Z}[\beta]$, and hence $\mathcal{O}_K[\beta] = \mathcal{O}_L$. Recall that $\beta$ satisfies the polynomial 
$g(x) = x^2 - (\alpha - c)x + 1 \in K[x]$. Then we can write $\disc(g) = [\mathcal{O}_L : \mathcal{O}_K[\beta]]^2 \cdot \disc(L/K) = 1 \cdot \disc(L/K) = \disc(L/K).$ Hence $\disc(L/K) = \disc(g) = (\alpha - c)^2 - 4$. 

We compute the field norm as the product of the conjugates:
\begin{align*}
N_{K/\mathbb{Q}}\big((\alpha-c)^2-4\big) &= \prod_{i=1}^n\big((\alpha_i-c)^2-4\big) \\
&= f(c+2)f(c-2).
\end{align*}
Hence the tower formula gives $\disc(L) = f(c+2)f(c-2)\cdot\disc(K)^2.$ Comparing with ~\eqref{eq:disc-L},
\[I_\alpha^4 \disc(K)^2 \cdot f(c+2)f(c-2) = f(c+2)f(c-2)\cdot\disc(K)^2.\]
Dividing both sides by $\disc(K)^2 \cdot f(c+2)f(c-2)$ gives $I_\alpha^4 = 1$, and hence $I_\alpha = 1$. Therefore $\mathcal{O}_K = \mathbb{Z}[\alpha]$ and $f$ is monogenic.
\end{proof}

\section{Examples}
\label{sec:example}
In this section, we list an example of a monogenic number field of degree $2^9=512$ generated using our construction.

We remind the reader of the notation $T_c$, where $F(x) = T_c(f)= x^nf(\tfrac 1 x +x+c)$ where $n$ is the degree of $f$, and $c$ is a fixed integer. We write $T_c(T_c(f)) = T_c^2(f)$, and so on for additional iterations.

We consider polynomials of the form $f(x) = x^4+nx+n$ for $n \in \mathbb{Z}$. A direct computation gives $f(-1) = 1$. Observe that with $c = 1$, we have $c-2 = -1$, so Remark~\ref{rem:eval} gives us
$$ T_{1}(f)(1) = f(3), \qquad T_{1}(f)(-1) = f(-1) = 1. $$ Since $f(c+2)f(c-2) = T_{1}(f)(1)\,T_{1}(f)(-1)$, the squarefree condition simplifies to squarefreeness of $T_{1}(f)(1) = f(3)$. We first show that $T^k_1(f)(-1) = 1$ for every $k$.
\begin{lem}
\label{lem:eval-one}
    For $f(x) = x^4 + nx + n$ with $c = 1$, we have $T^k_1(f)(-1) = 1$ for all $k \geq 0$.
\end{lem}
\begin{proof}
    The base case $f(-1) = 1$ is immediate. For the inductive step, assume that $T^{k-1}_1(f)(-1) = 1$. Since $T^k_1(f)(-1) = (-1)^{\deg {T^{k-1}_1(f)}} T^{k-1}_1(f)(-1)$ by construction, and $\deg T^{k-1}_1(f) = 4 \cdot 2^{k-1}$ is divisible by 4, we have $(-1)^{\deg T^{k-1}_1(f)} = 1$. Therefore $T^k_1(f)(-1) =  1$.
\end{proof}
Using a computer algebra system, one can easily find specific values of $n$ for which $f$ is irreducible and monogenic. For example, the values $n \in \{-5, -3, -2, -1, 1, 2, 6, 7\}$ all give monogenic fields. We take $n = 1$ and $c=1$ for simplicity, which gives $f = x^4+x+1$ with $\disc(f) = 229$ and with $f(3) = 85 = 5 \cdot 17$ squarefree. Since $f$ has degree 4, after $k=7$ iterations, we obtain a field of degree $4 \cdot 2^7 = 512$ (assuming the transformed polynomial is irreducible at each iteration) which is a standard cryptographic parameter size. Recall that to verify monogenicity of this field, we need to check that $T_1^1(f)(3), \ldots,T^6_1(f)(3)$ are all squarefree. In particular, the monogenicity of the degree 512 field follows from the squarefreeness of $T^6_1(f)(3)$. We verified this computationally using SageMath~\cite{sagemath}. Squarefreeness was checked by complete integer factorization of each $T_1^k(f)(3)$.

\begin{center}
\begin{tabular}{c c c c}
\toprule
$k$ & $\deg T_1^k(f)$ & \#digits of $T_1^k(f)(3)$ & squarefree \\
\midrule
1 & 8   & 5   & \checkmark \\
2 & 16  & 10  & \checkmark \\
3 & 32  & 20  & \checkmark \\
4 & 64  & 39  & \checkmark \\
5 & 128 & 77  & \checkmark \\
6 & 256 & 154 & \checkmark \\
\bottomrule
\end{tabular}
\end{center}

The first few values are
\[
\begin{aligned}
T_1^1(f)(3) &= 28993 \\
T_1^2(f)(3) &= 2232559585 \\
T_1^3(f)(3) &= 10649905326423671233,
\end{aligned}
\]
each of which is squarefree (see Appendix~\ref{app:polys} for the full factorization).
\begin{rem}
The same squarefreeness also certifies irreducibility. For $k \ge 1$, we have $|T_1^k(f)(1)|  = T_1^{k-1}(f)(3)$ and 
$|T_1^k(f)(-1)| = 1$ by Lemma~\ref{lem:eval-one}. Since $T_1^{k-1}(f)(3)$ 
is squarefree and not $\pm 1$, it is not a perfect square, so Theorem~\ref{theorem:main-irreducibility} condition (1) makes $T_1^k(f)$ irreducible. Inductively from the base $f$, this gives irreducibility at every level up to $T_1^7(f)$.
\end{rem}
Therefore, $T^7_1(f)$ generates a monogenic number field of degree $512$. Similar analysis can be performed on $f$ for other values of $n$; however, we believe such computations to become increasingly expensive for larger values of $n$.

\section{Open Questions}
\label{sec:open}
It is widely conjectured that a polynomial $f \in \mathbb{Z}[x]$ with no repeated roots and no fixed square divisor takes infinitely many squarefree values; Granville \cite{granville1998abc} showed this follows from the ABC conjecture. We therefore propose the following conjecture.

\begin{conj}
Let $f\in\mathbb{Z}[x]$ be monic, irreducible, and monogenic, and set $g(x)=f(x+2)f(x-2)$. Suppose $g$ has no repeated roots and no fixed square divisor, and that $T_c(f)$ is irreducible for infinitely many $c$. Then $T_c(f)$ generates a monogenic number field of degree $2n$ for infinitely many $c$. In particular, the transform produces infinitely many monogenic number fields of degree $2n$.
\end{conj}

We leave the reader with the following questions.
\begin{enumerate}
\item Our monogenicity criterion requires checking that $f(c+2)f(c-2)$ is squarefree. It remains an open problem in algorithmic number theory to find a simple procedure to check whether an integer is squarefree or not \cite{complexity}. Are there efficient algorithms to check squarefreeness without needing factorization?\footnote{Booker, Hiary, and Keating \cite{bookerhiarykeating2015} give an algorithm to prove an integer is squarefree without factoring, conditional on the Generalized Riemann Hypothesis. Apart from this, we are not aware of any progress on this problem.}

\item For which base polynomials and values of $c$ does the transform produce Galois extensions? Can one characterize when $\mathbb{Q}(\beta)$ is Galois directly in terms of $f$ and $c$?

\item Conditional on the ABC conjecture, our transform produces 
infinitely many monogenic fields of each degree $2^kn$ for each $k$. 
Can this be proved unconditionally?
\end{enumerate}

\bibliography{monogenic}{}
\bibliographystyle{alpha}

\appendix
\section{Explicit integer evaluations for the degree-512 tower}
\label{app:polys}
Recall $f(x)=x^4+x+1$ and $c=1$. Monogenicity of $T_1^k(f)$ at each level
follows from the squarefreeness of $T_1^{k-1}(f)(3)$, verified in
SageMath~\cite{sagemath} by complete integer factorization. We record the
evaluations $T_1^k(f)(3)$ below.

$T_1^1(f)(3) = 28993$

$T_1^2(f)(3) = 2232559585$

$T_1^3(f)(3) = 10649905326423671233$

$T_1^4(f)(3) = 211634739909134148439038192622160909185$

For $k=5,6$ the values have $77$ and $154$ decimal digits:

$ T_1^5(f)(3) = \seqsplit{76149772278514573380264313010048717269966292252651636717168679489946080667393} $

$ T_1^6(f)(3) = \seqsplit{9210200409588115967821453482497876632413382320180805400109658859589311158421746780705086283420230400187252209650480048817017657607843223069506588383235585} $

We now give the full factorization of these integers from SageMath to establish that they are squarefree.

$T_1^1(f)(3) = 79 \cdot 367$

$T_1^2(f)(3) = 5 \cdot 11437 \cdot 39041$

$T_1^3(f)(3) = 19 \cdot 333323 \cdot 1681616129009$

$T_1^4(f)(3) = 5 \cdot 13921 \cdot 3040510594197746547504319985951597$

$T_1^5(f)(3) = 23 \cdot 18859 \cdot 135043 \cdot 1082987312549587 \cdot \seqsplit{1200401878325723969744954187178965938599958281402589}$

$T_1^6(f)(3) = 5 \cdot 103435369 \cdot 9434375759 \cdot 2401020462143 \cdot \seqsplit{17516469089252729797} \cdot p,$
where $p = \seqsplit{44882228243845612112027161806541341327947594224808218347412922957193041078034608366905827295622491336337}$.
\end{document}